\documentclass[11pt]{amsart}
\usepackage[dvipsnames]{xcolor}
\usepackage[left=1in,top=1in,right=1in,bottom=1in]{geometry}
\usepackage{wasysym, mathtools, amssymb, tikz, tikz-cd, mathrsfs}
\usepackage[T1]{fontenc}
\usepackage[colorlinks=true, linkcolor=red!80!black, urlcolor=purple, citecolor=blue!70!black]{hyperref}

\usetikzlibrary{matrix,arrows}

\input xy 
\xyoption{all} %

\newtheorem{question}{Question}
\newtheorem{theorem}{Theorem}[section]
\newtheorem{lemma}[theorem]{Lemma}
\newtheorem{corollary}[theorem]{Corollary}
\newtheorem{proposition}[theorem]{Proposition}
\newtheorem{conjecture}[theorem]{Conjecture}

\theoremstyle{definition}
\newtheorem{definition}[theorem]{Definition}

\newtheorem{remark}[theorem]{Remark}

\def\owlset@enDeligne-Mumfordark{\ensuremath{\scriptscriptstyle\blacksquare}}

\newcommand*{\DashedArrow}[1][]{\mathbin{\tikz [baseline=-0.25ex,-latex, dashed,#1] \draw [#1] (0pt,0.5ex) -- (1.3em,0.5ex);}}

\newcommand{\Aff}{{\mathbb A}}
\newcommand{\C}{{\mathbb C}}

\newcommand{\PP}{{\mathbb P}}
\newcommand{\Q}{{\mathbb Q}}
\newcommand{\R}{{\mathbb R}}

\def\bbar#1{\setbox0=\hbox{$#1$}\dimen0=.2\ht0 \kern\dimen0 
\overline{\kern-\dimen0 #1}}

\newcommand{\OO}{{\mathscr O}}

\DeclareMathOperator{\Pic}{Pic}

\begin{document}
\title{THE ERD\H{o}S-ULAM PROBLEM, LANG'S CONJECTURE, AND UNIFORMITY}
\author{Kenneth Ascher}
\email{kascher@princeton.edu}
\author{Lucas Braune}
\email{lvhb@uw.edu}
\author{Amos Turchet}
\email{amos.turchet@sns.it}

\begin{abstract}
A rational distance set is a subset of the plane such that the distance between any two points is a rational number.
We show, assuming Lang's Conjecture, that the cardinalities of rational distance sets in general position are uniformly bounded, generalizing results of Solymosi-de Zeeuw, Makhul-Shaffaf, Shaffaf, and Tao. In the process, we give a criterion for certain varieties with non-canonical singularities to be of general type.
\end{abstract}

\maketitle

\section{Introduction}\label{sec:intro}
A \emph{rational distance set} is a subset $S$ of $\R^2$ such that the distance between any two points of $S$ is a rational number. In 1946, Ulam posed the following, based on a result of Anning-Erd\H{o}s \cite{ae}:

\begin{question}
Is there a rational distance set that is dense for the Euclidean topology of $\R^2$?
\end{question}

While this problem is still open, Shaffaf \cite{shaffaf} and Tao \cite{tao} independently showed that Lang's Conjecture (Conjecture \ref{conj:lang} below) implies that the answer to the Erd\H os-Ulam question is `no'.
In fact, they showed that if Lang's Conjecture holds, a rational distance set cannot even be dense for the Zariski topology of $\R^2$, i.e. must be contained in a finite union of real algebraic curves. Solymosi and de Zeeuw \cite{sz} proved (unconditionally), using Faltings' proof of Mordell's conjecture, that a rational distance set contained in a real algebraic curve must be finite, unless the curve has a component which is either a line or a circle.
Furthermore, any line (resp. circle) containing infinitely many points of a rational distance set must contain all but at most four (resp. three) points of the set. 

One can rephrase the result of \cite{sz} as almost all points of an infinite rational distance set contained in a union of curves tend to concentrate on a line or circle. It is therefore natural to consider the ``generic situation'', and so motivated by \cite{sz}, we say a subset $S\subseteq \R^2$ of cardinality $n$ is in \emph{general position} if no subset of $S$ of cardinality $n-4$ is contained in a line and no subset of $S$ of cardinality $n-3$ is contained in a circle.

For example, a set of seven points in $\R^2$ is in general position if and only if no $7-4 = 3$ points lie on a line and no $7-3 = 4$ points lie on a  circle.

Before stating our main result, we note that assuming Lang's conjecture, rational distance sets in general position are necessarily finite, by the above mentioned results \cite{sz, shaffaf, tao}.

\begin{theorem}\label{thm:mainthm}
Assume Lang's Conjecture. There exists a uniform bound on the
  cardinality of a rational distance set in general position.
\end{theorem}

Erd\H os famously asked whether there exists a rational distance set with seven elements in general position. To stress the parallel between Lang's Conjecture and Erd\H os' question, we emphasize the following consequence of our main result.

\begin{corollary}\label{cor:Lang}
If there exist rational distance sets in general position of cardinality larger than any fixed constant, then Lang's Conjecture does not hold.
\end{corollary}

Erd\H{o}s' question was answered in the affirmative by Kreisel and Kurz \cite{kk}. We are unaware of any examples of rational distance sets in general position of cardinality larger than seven. We do note, however,  that for sets of cardinality greater than seven, our notion of general position is strictly weaker than the one appearing previously e.g. in work of  \cite{kk}. For them, general position means no three points contained on a line, and no four on any circle.

Finally, we note that Pasten proved that the ABC conjecture also implies a negative answer to the Erd\H{o}s-Ulam problem \cite{pasten}.

\subsection{Method of proof} To prove Theorem \ref{thm:mainthm}, we note that points of rational distance sets lift to rational points of curves and surfaces of general type, (see \cite{sz, tao}).
We then show that Lang's Conjecture implies uniform bounds for these sets of rational points, using results of Caporaso-Harris-Mazur \cite{chm} and Hassett \cite{hassett}. The idea of using uniformity to study rational distance sets first appeared in the paper of Makhul and Shaffaf \cite{ms}, who used  uniformity for curves \cite{chm}. To obtain our Theorem \ref{thm:mainthm} (which generalizes \cite{ms}), we use both uniformity for curves \cite{chm} as well as uniformity for surfaces \cite{hassett}.

Our proof of Theorem \ref{thm:mainthm} uses a result that certain (singular) complete intersections of four quadrics in $\PP^6$ are of general type. This result was first shown by Tao \cite{tao} (see also \cite{SK}). We include in this paper a more general statement (Proposition \ref{prop:criterion-general-type} below), which implies Tao's result and we think may be of interest in its own right. 

\subsection*{Acknowledgments}
We learned about this problem from Tao's blog post \cite{tao} and a comment of J\'ozsef Solymosi. We thank S\'andor Kov\'acs and Bianca Viray for useful conversations, and J\'ozsef Solymosi and Sascha Kurz for helpful comments on a draft of this paper. K.A. is supported in part by an NSF postdoctoral fellowship. A.T. is supported in part by funds from NSF grant DMS-1553459. We thank the referees for their many helpful comments which greatly improved this paper.

\section{Background}

Faltings proved in \cite{faltings} that a curve of genus at least two can have only finitely many points with rational coordinates (or more generally with coordinates in a number field). Varieties of general type provide the natural generalization of curves of genus at least two to higher dimensional varieties. 

\begin{definition}\label{def:generaltype} A projective variety $X$ is of general type if for any desingularization $\widetilde{X} \to X$ the canonical divisor $\omega_{\widetilde{X}}$ is big. \end{definition}

 In Lemma \ref{lemma:can-gen-type} and Proposition \ref{prop:criterion-general-type} we will see criteria for a variety to be of general type. Using this notion, Lang (and Bombieri for surfaces) proposed a generalization of Faltings' theorem for varieties of higher dimension.

\begin{conjecture}[Lang's Conjecture]\label{conj:lang} Let $X$ be a projective variety of general type defined over a number field $K$. Then the set of rational points $X(K)$ is not Zariski dense in $X$.\end{conjecture}

Conjecture \ref{conj:lang} is known in full generality only for curves, where it reduces to Faltings' Theorem \cite{faltings} and subvarieties of abelian varieties \cite{faltings_ab}. The conjecture of Lang has a number of striking implications. One of the more spectacular ones is the following theorem of Caporaso-Harris-Mazur.

\begin{theorem}[{\cite[Theorem 1.1]{chm}}]
\label{thm:chm}
Assume Conjecture \ref{conj:lang}.
For each number field $K$ and positive integer $g\ge 2$, there exists an integer $B(K, g)$ such that no smooth curve of genus $g$ defined over $K$ has more than $B(K, g)$ rational points.
\end{theorem}

Following Caporaso-Harris-Mazur, Hassett proved the following implication of the Lang conjecture for surfaces that will be used in this paper \cite{hassett}.

\begin{theorem}[{\cite[Theorem 6.2]{hassett}}]
\label{thm:hassett}
Assume Conjecture \ref{conj:lang}.
Let $X\to B$ be a flat family of projective surfaces defined over a number field $K$, such that the general fiber is an integral surface of general type.
For any $b\in B(K)$ for which $X_b$ is of general type, let $N(b)$ be the sum of the degrees of the components of $\overline{X_b(K)}$. Then $N(b)$ is bounded by a constant independent of $b$.
\end{theorem}

Analogous statements were proven for (stably) integral points in \cite{Aell} and \cite{AdVT}; see \cite{ATnotes} for a survey of these and related results.

\section{Uniform bounds for rational distance sets}

The goal of this section is to give a proof of Theorem \ref{thm:mainthm}, modulo the claim that a certain surface is of general type, which will be proved as Proposition \ref{prop:surf-gen-type} below. We begin with a well known lemma that says that rational distance sets may be brought to a standard form after translations, rotations and rational dilations. For a proof, see \cite[Lemma 2]{shaffaf}.

\begin{lemma}
\label{lem:lattice}
Let $S\subseteq \R^2$ be a rational distance set.
If $S$ contains $(0,0)$ and $(1,0)$, then there exists an integer $k > 0$ such that $S \subseteq \{(a,b\sqrt k)~|~a,b\in \Q\}$.
\end{lemma}

We will let $x,y,z$ be the coordinates on $\PP^2_\C$ and identify $\R^2$ with the set of real points of the affine open subset of $\PP_\C^2$ where $z=1$. The following lemma will be crucial in the proof of Lemma \ref{lemma:sz}.

\begin{lemma}
\label{lemma:choice-of-points}
Let $C\subset \PP^2_\C$ be an irreducible curve of degree $d$ defined over $\R$.
Let $S\subseteq \R^2$ be a subset.
Given $P := (a,b)\in \R^2$, let $\ell_P\subseteq \PP^2_\C$ be the complex line with equation $(x-az) + i(y-bz)=0$, and let $\bar\ell_P$ be its conjugate.
Suppose that either
\begin{enumerate}
\item $d=1$ and $\#(S\setminus C)\ge 5$; or
\item $d\ge 3$ and $\#(C\cap S)\ge d(d-1) + \tfrac 5 2 d + 1$.
\end{enumerate}
Then there exist three points $P_1, P_2, P_3\in S$ such that the intersection of $C$ with the union of six lines
$\cup_{j=1}^3 \ell_{P_j} \cup \bar\ell_{P_j}$
is transverse at $\geq \max(3(d-2), 6)$ points. In particular, the intersection is everywhere transverse if $d=1$.
\end{lemma}

We note that, while the lower bound $\max(3(d-2), 6)$ in the conclusion of Lemma \ref{lemma:choice-of-points} is natural in the context of our proof of this lemma, it is stronger than required by its application below, namely the proof of Lemma \ref{lemma:sz}.
For this application, a lower bound of $6$ would suffice.

\begin{proof}[Proof of Lemma \ref{lemma:choice-of-points}]
We distinguish the two cases in the statement of the proposition.
First, suppose that $d=1$ and $\#(S\setminus C)\ge 5$.
Given a point $P_a \in \R^2\setminus C$, there exists at most one ``bad'' point $P_b \in \R^2\setminus C$ such that the intersection of $C$ with 
\begin{equation*}
\ell_{P_a}\cup \bar\ell_{P_a} \cup \ell_{P_b} \cup \bar\ell_{P_b}
\end{equation*}
is not transverse, namely $P_b$ is the reflection of $P_a$ with respect to $C$.
For this bad point, the intersection consists of two (rather than four) points. To construct the triple $(P_1, P_2, P_3)$ we pick a point $P_1\in S\setminus C$, then we pick a point $P_2\in S\setminus C$ different from $P_1$ and its reflection and, finally, we pick a point $P_3\in S\setminus C$ different from $P_1$, $P_2$ and their reflections.

Now suppose that $d\ge 3$ and $\#(C\cap S)\ge d(d-1) + \tfrac 5 2 d + 1$, and let $Q = (-i, 1, 0)$ and $\overline Q = (+i,1,0)$.
Let $\mu\ge 0$ be the multiplicity of $C$ at $Q$, which is also the multiplicity at $\overline Q$, since $C$ is defined over $\R$.
The intersection of $C$ with the line joining $Q$ and $\overline Q$ has degree $d$, so $\mu\le d/2$.
This implies that there are at most $d/2$ points $P\in S$ such that $\ell_P$ is tangent to $C$ at $Q$, since each line through $Q$ contains at most one point of $\R^2$.

Let $C^\vee \subseteq (\PP^2_\C)^\vee$ be the dual curve parametrizing lines in $\PP^2_\C$ that are either tangent to $C$ at a smooth point, or pass through a singular point of $C$.
Let $m\subseteq (\PP_\C^2)^\vee$ be the line parametrizing lines through $Q$.
It is a standard fact that $C^\vee$ has degree at most $d(d-1)$, so that the intersection $\overline{C^\vee\setminus m}$ with $m$ consists of at most $d(d-1)$ points.
Thus there are at most $d(d-1)$ points $P\in S$ such that $\ell_P$ is tangent to $C$ at a point different from $Q$.

Let $G$ be the set of points $P\in C\cap S$ such that the line $\ell_P$ is not tangent to $C$ at $Q$ and meets the rest of $C$ transversely at smooth points.
For each $P\in G$, the set $\delta_P = (\ell_P\setminus \{Q, P\})\cap C$ consists of $d-\mu-1$ distinct points. 
We have $d-\mu-1 \ge \max(\tfrac 1 2 (d-2), 1)$, so it suffices to find $P_1, P_2, P_3\in G$ such that the sets $\delta_{P_1}, \overline\delta_{P_1}, \delta_{P_2}, \overline\delta_{P_2}, \delta_{P_3}, \overline\delta_{P_3}$ are pairwise disjoint, where $\overline{\delta}_{P} = (\bar{\ell}_p \setminus  \{\overline{Q},P\}) \cap C$.
By construction $\delta_{P_i} \cap \delta_{P_j} = \emptyset$, $\delta_{P_i} \cap \overline{\delta}_{P_i} = \emptyset$ and $\overline{\delta}_{P_i} \cap \overline{\delta}_{P_j} = \emptyset$ for all distinct $P_1, P_2, P_3\in G$ and all $i \neq j$.
Moreover, by the preceding two paragraphs, the set $G$ has cardinality at least $2d+1$.

This allows us to construct  $(P_1, P_2, P_3)$ as follows: we pick $P_1\in G$, then we pick $P_2\in G$ such that $\bar\ell_{P_2}$ is disjoint from $\delta_{P_1}$ and, finally, we pick $P_3\in G$ such that $\bar\ell_{P_3}$ is disjoint from $(\delta_{P_1}\cup \delta_{P_2})$.
\end{proof}

Some parts that follow summarize some of the work of \cite[3, 4]{sz} -- we give conditions under which points of a rational distance set $S$ contained in a real algebraic curve $C$ can be lifted to rational points of a curve $D$ of genus at least two. We additionally give a lower bound for $\#(C\cap S)$ and upper bound for $g(D)$.
The fact that these bounds can be given in terms of the degree of the curve is important for the proof of Proposition \ref{prop:curves} below.

\begin{lemma}[see also {\cite{sz}}]
\label{lemma:sz}
Let $S\subset \R^2$ be a rational distance set and let $C\subset \PP^2_\C$ be an irreducible projective curve of degree $d\ne 2$ defined over $\R$.
If $d=1$, suppose that $\#(S\setminus C) \ge 5$;
otherwise, suppose that $\#(C\cap S) \ge d(d-1) + \tfrac 5 2 d + 1$ and that $C\cap S$ is contained in the smooth locus of $C$.
Then there exists a smooth projective curve $D$ defined over $\Q$ whose genus satisfies $2 \le g(D) \le d^2 + 1$, and a morphism $\pi : D\to C$  defined over $\R$ such that $(C\cap S) \setminus \pi(D(\Q))$ consists of at most 3 points.
\end{lemma}

\begin{proof}
By Lemma \ref{lem:lattice}, we may assume that $S\subseteq \{(a,b \sqrt k)\in \R^2~|~a,b\in \Q\}$ for some $k \in \mathbb{Z}_{>0}$.
Let $\phi : \PP_\C^2\to \PP_\C^2$ be the linear automorphism sending $(x,y,z)\mapsto(x,y/\sqrt k,z)$. Let $C' := \phi(C)$ and $S' := \phi(S)$. Then $C'$ is defined over $\Q$, as it contains $\geq$ $d^2+1$ points with rational coordinates in $\PP_\C^2$, namely the points of $S'\cap C'$.
Let $f\in \Q[x,y,z]$ be a homogeneous equation for $C'$ of degree $d$.

Given $P=(a,b)\in S'$, let $q_P$ denote the quadratic polynomial $(x-az)^2 + k (y-bz)^2\in \Q[x,y,z]$.
The zero locus of $q_P$ in $\PP_\C^2$ is a pair of complex lines that meet at $P$.
By Lemma \ref{lemma:choice-of-points}, we can find points $P_1, P_2, P_3\in S'$ such that the union of six lines $\{q_{P_1} q_{P_2} q_{P_3}=0\} \subseteq \PP_\C^2$ meets $C'$ transversely at $\geq$ 6 points.
Let $x,y,z,w$ be the coordinates of the weighted projective space $\PP(1,1,1,3)$, and
let $D'$ be the curve in $\PP(1,1,1,3)$ defined by $w^2 - q_{P_1}q_{P_2}q_{P_3} = f = 0$.
Let $\pi' : D'\to C'$ be the projection onto the first three coordinates. Then $\pi'(D'(\Q))$ contains $C'\cap S'$.

Let $\pi'' : D \to C'$ be the composition of the normalization $D\to D'$ with $\pi'$, and let $\pi := \phi^{-1}\circ \pi''$.
We note that $\pi''$ is a double cover whose ramification divisor is reduced and consists of the points of odd multiplicity in the intersection $\{q_{P_1} q_{P_2} q_{P_3}=0\}\cap C'$.
This ramification divisor is nonempty (in fact, contains at least six points), so $D$ is connected.
Applying Riemann-Hurwitz, we conclude that 
$g(D) = 2g(C) -1 + \frac{r}{2},$ where $6 \leq r \leq 6d$. As $C$ is a plane curve, we have that $0 \leq g(C) \leq \frac{1}{2}(d-1)(d-2)$, and therefore $2 \leq g(D) \leq d^2+1$.

The set $T = (C'\cap S')\setminus \{P_1, P_2, P_3\}$ is contained in the smooth locus of $C'$ by assumption and disjoint from the ramification locus of $\pi'$. Therefore the inverse image of every point in $T$ under $\pi'$ consists of two smooth points of $D'$ defined over $\Q$.

Thus $(C'\cap S')\setminus \pi''(D(\Q))$ is contained in $\{P_1, P_2, P_3\}$, and therefore $(C\cap S)\setminus \pi(D(\Q))$ consists of at most three points.
\end{proof}

In the following proposition, we show that Lang's Conjecture implies a stronger version of the results of \cite{sz} mentioned in the introduction.
We say that an algebraic curve $C\subset \PP_\C^2$ is a \emph{circle} if it has degree two and the intersection $C\cap \R^2$ is a circle in the usual sense.

\begin{proposition}
\label{prop:curves}
Assume Lang's Conjecture.
Let $S\subset \R^2$ be a rational distance set and let $C \subset \PP^2_\C$ be an irreducible projective curve of degree $d$. Assume
\begin{enumerate}
\item $C$ is neither a line nor a circle; or
\item $C$ is a line and $S$ contains at least five points not in $C$; or
\item $C$ is a circle and $S$ contains at least four points not in $C$;
\end{enumerate}
then the cardinality of $C\cap S$ is bounded by a constant that depends only on $d$.
\end{proposition}

\begin{proof}
Suppose first that $d\ne 2$.
To prove the result in this case, we may assume that $C\cap S$ is contained in the smooth locus of $C$, since the cardinality of the singular locus is bounded by the arithmetic genus $\tfrac 1 2 (d-1)(d-2)$.
We further may assume that $\#(C\cap S)\ge d(d-1) + \tfrac 5 2 d + 1$; otherwise there is nothing to prove.
Since $C$ and $S$ satisfy the hypotheses of Lemma \ref{lemma:sz}, the result follows by the theorem of Caporaso-Harris-Mazur (Theorem \ref{thm:chm}) applied to the curve $D$ of Lemma \ref{lemma:sz}.

Now suppose that $d=2$. Let $P\in C\cap S$ be a point and let $\phi : \R^2\setminus \{P\}\to \R^2$ be the inversion with respect to the unit circle centered at $P$.
Let $S' := \phi(S\setminus \{P\})\cup \{ P \}$; then $S'$ is a rational distance set.
Let $\Phi : \PP^2_\C\DashedArrow[->,densely dashed] \PP_\C^2$ be the unique birational involution that agrees with $\phi$ on the Zariski-dense subset $\R^2\setminus \{P\} \subseteq \PP^2_\C$. For example, if $P=(0,0)$, then $\Phi$ is given by $\Phi(x,y,z) = (xz,yz,x^2+y^2)$.
Let $C' \subseteq \PP_\C^2$ be the image of $C$ under $\Phi$.
If $C$ is a circle, then $C'$ is a line and $\#(S'\setminus C)\ge 5$; otherwise, $C'$ is a cubic curve.
By the preceding paragraph applied to $(C',S')$, the cardinality of $C'\cap S'$ is bounded by a constant independent of $C'$ and $S'$.
\end{proof}

\begin{remark}\label{rmk:generalposition}
If $S \subseteq \R^2$ is a rational distance set in general position, then the hypotheses of Proposition \ref{prop:curves} are satisfied for every curve $C\subseteq \PP_\C^2$.
\end{remark}

The following result strengthens the theorem of Shaffaf \cite{shaffaf} and Tao \cite{tao} that rational distance sets are contained in real algebraic curves.

\begin{proposition}
\label{prop:bounded-degree}
Assume Lang's Conjecture. Given any rational distance set, there exists a Zariski-closed proper subset $Z$ of $\PP^2_\C$ containing it.
Moreover, there exists a positive integer $d$ such that, for each rational distance set, the subset $Z$ may be chosen with the sum of the degrees of its irreducible components no greater than $d$.
\end{proposition}

\begin{proof}
Let $S\subset \R^2$ be a rational distance set.
By Lemma \ref{lem:lattice}, we may assume that $S\subseteq \{(a,b \sqrt k)\in \R^2~|~a,b\in \Q\}$ for some positive integer $k$, and we may also assume that $\# S \geq 4$.
Let $\phi : \R^2\to \R^2$ be the linear automorphism that sends $(x,y)\mapsto(x,y/\sqrt k)$. Let $S' := \phi(S)$. 
Let $(a_1,b_1),\dotsc,(a_4,b_4) \in S'$ be four distinct points, let $x,y,z,r_1,r_2,r_3,r_4$ be the coordinates of $\PP^6_\C$, and let $V\subseteq \PP^6$ be the intersection of the four quadrics with equations $r_j^2 = (x-a_j z)^2 + k (y-b_jz)^2$, where $j=1,2,3,4$. 

If $\pi : V\to \PP^2_\C$ is the projection onto the first three coordinates, then $\pi(V(\Q))$ contains $S'$.
Indeed, if $(u,v)\in S'$ then, for each $j=1,2,3,4$, the two points $(u,v\sqrt k)$ and $(a_j,b_j\sqrt k)$ are both in $S$, so the distance $s_j := \sqrt{(u-a_j)^2 + k(v-b_j)^2}$ is a rational number. Therefore, we conclude that $(u,v,1,s_1,s_2,s_3,s_4)\in V(\Q)$.

We will see in Proposition \ref{prop:surf-gen-type} below that the surface $V$ is of general type, and so by Lang's Conjecture the Zariski-closure $\overline {V(\Q)}$ of the set of rational points is a proper subset of $V$.
By Hassett's result (Theorem \ref{thm:hassett}) applied to the family of all complete intersections of four quadrics in $\PP^6_\C$, there exists a constant $d$, independent of $S$ and the choice of the four points $(a_i,b_i)$, such that the sum of the degrees of the irreducible components of $\overline {V(\Q)}$ is at most $d$.

The morphism $\pi : V\to \PP^2_\C$ is finite, so $\pi(\overline {V(\Q)})$ is a Zariski-closed proper subset $Z \subset \PP^2_\C$, the sum of the degrees of the irreducible components of $Z$ is at most $d$, and $Z$ contains $S'$.
\end{proof}

We now combine Propositions \ref{prop:curves} and  \ref{prop:bounded-degree} to establish the main result of this paper.

\begin{proof}[Proof of Theorem \ref{thm:mainthm}]
Let $S$ be a rational distance set in general position.
By Proposition \ref{prop:bounded-degree}, there exists an integer $d$, independent of $S$, and a Zariski-closed proper subset $Z\subset \PP^2_\C$ that contains $S$ and such that the sum of the degrees of the irreducible components of $Z$ is at most $d$.
Thus $Z$ has at most $d$ irreducible components, each of which has degree at most $d$.
The components of dimension zero each contain at most one point of $S$.
By Proposition \ref{prop:curves} since the set $S$ is in general position (see Remark \ref{rmk:generalposition}), there exists a constant depending only on $d$, and hence independent of $S$, that bounds the number of points of $S$ contained in each of the one-dimensional irreducible components of $Z$.
\end{proof}

\section{A criterion for general type}
\label{sec:gen_type}

The main result of this section is Proposition \ref{prop:criterion-general-type}, which gives a criterion for varieties with non-canonical singularities to be of general type.
We begin with a definition.

\begin{definition}
Let $X$ be a normal variety whose canonical divisor $K_X$ is $\Q$-Cartier. Let $f: \widetilde X\to X$ be a resolution of singularities with irreducible exceptional divisors $E_1,\dotsc,E_r$, and let $a_1,\dotsc,a_r\in \Q$ be the corresponding discrepancies (see \cite[Definition 2.22]{kollar-mori}), so that $K_{\widetilde X} = f^* K_X + \sum_j a_j E_j$ in $\Pic(\widetilde X)\otimes \Q$.
We say that $X$ has \emph{canonical singularities} if $a_j \ge 0$ for all $j=1,\dotsc,r$.
\end{definition}

We note that the above definition is independent of the choice of resolution.
One can check whether a variety with canonical singularities is of general type without passing to a resolution.
This observation is formalized by the following lemma.

\begin{lemma}
\label{lemma:can-gen-type}
Let $X$ be a normal projective variety over $\C$.
Suppose that the canonical divisor $K_X$ is $\Q$-Cartier.
If $X$ is of general type, then the canonical divisor $K_X$ is big.
The converse holds if $X$ has canonical singularities.
\end{lemma}

\begin{proof}
Let $\omega_X := \OO_X(K_X)$ be the canonical sheaf of $X$.
Thus a section of $\omega_X$ is a section of the top exterior power of the cotangent bundle of the smooth locus of $X$.
Let $\rho : \widetilde X\to X$ be a resolution of singularities.
For each integer $l > 0$, there is a natural inclusion $\tau_l : \rho_* \omega_{\widetilde X}^{\otimes l}\hookrightarrow \omega_X^{[l]}:=(\omega_X^{\otimes l})^{\vee\vee}$ given by restriction of pluricanonical forms to the complement of the exceptional divisor of $\rho$.

Furthermore, if $X$ has canonical singularities, then $\tau_l$ is an isomorphism for each $l$ such that $\omega_X^{[l]}$ is locally free.
Indeed, by definition of canonical singularities, pullbacks of sections of $\omega_X^{[l]}$ do not have poles when regarded as rational sections of $\omega_{\tilde X}^{\otimes l}$; pullback of sections defines the inverse of $\tau_l$.
\end{proof}

Proposition \ref{prop:criterion-general-type} partially generalizes the converse in the above lemma (``partially'' because in this proposition we assume $K_X$ is ample).
Before stating it, we recall the definitions of Hilbert-Samuel multiplicity and ordinary multiple point.

\begin{definition}
Let $X$ be an algebraic scheme over $\C$.
Let $x\in X$ be a closed point.
\begin{enumerate}

\item There exists a unique polynomial $H\in \Q[l]$ such that $H(p) = \dim_{\C} \OO_{X,x}/\mathfrak m_x^p$ for all sufficiently large $p$.
The degree of $H$ is equal to the Krull dimension $d$ of $\OO_{X,x}$.
The \emph{Hilbert-Samuel multiplicity} of the local ring $\OO_{X,x}$ is equal to $d!$ times the leading coefficient of $H$.
It is a positive integer that we denote by $e(\OO_{X,x})$.

\item Let $\pi : B\to X$ be the blowup of $X$ at the ideal sheaf of $x$. Let $E\subseteq B$ be the exceptional divisor of $\pi$.
If $E$ is smooth, we say that $x\in X$ in an \emph{ordinary multiple point}.

\end{enumerate}
\end{definition}

In terms of the canonical isomorphism $E = \operatorname{Proj}(\oplus_{l\ge 0}\mathfrak m^l/\mathfrak m^{l+1})$, we have $\OO_E(1)=\OO_B(-E)|_E$.
By asymptotic Riemann-Roch, the Hilbert-Samuel multiplicity $e(\OO_{X,x})$ is equal to the degree of $E$ with respect to this ample invertible sheaf. If $X$ has an ordinary multiple point at $x$, then the blowup $B$ is smooth in a neighborhood of $E$, and so ordinary multiple points are isolated singularities.

The following criterion extends Lemma \ref{lemma:can-gen-type} when the set of non-canonical singularities of the variety consists of ordinary multiple points. 

\begin{proposition}
\label{prop:criterion-general-type}
Let $X$ be a normal projective variety of dimension $d$ over $\C$ whose canonical disivor $K_X$ is $\Q$-Cartier.
Suppose that the set $\Sigma$ of non-canonical singularities of $X$ consists of ordinary multiple points.
For each $x\in \Sigma$, let $E_x$ denote the exceptional divisor of the blowup of $X$ at $x$, and let $a(E_x,X)$ denote the discrepancy of $E_x$ with respect to $X$.
If $K_X$ is ample and 
\begin{equation}
\label{K-positive}
K_X^d > \sum_{x\in \Sigma} |a(E_x,X)|^d \cdot  e(\OO_{X,x}),
\end{equation}
then $X$ is of general type.
\end{proposition}

\begin{proof}
Let $\pi :B\to X$ be the blowup of $X$ at the points of $\Sigma$.
By definition of an ordinary multiple point, for each $x\in \Sigma$, the exceptional divisor $E_x$ is smooth, and hence $B$ is smooth in a neighborhood of $E_x$.
Thus $B$ has canonical singularities.
By Lemma \ref{lemma:can-gen-type} it suffices to show that $\omega_B$ is big.

For each non-canonical singularity $x\in \Sigma$, write  $a_x := a(E_x,X)\in \Q_{<0}$.
Thus $K_B = \pi^* K_X + \sum_{x\in \Sigma} a_x E_x$ in $\mathrm{Pic}(X)\otimes \Q$.
In fact, let $l_0$ be the smallest positive integer $l$ such that $\omega_X^{[l]}:= (\omega_X^{\otimes l})^{\vee\vee}$ is locally free. 
Then $l_0 a_x$ is an integer for each $x\in \Sigma$, and $\omega_B^{\otimes l}(-l \sum_{x\in \Sigma} a_x E_x) = \pi^* \omega_X^{[l]}$ as subsheaves of $\omega_B^{\otimes l}\otimes k(B)$ for all positive integers $l$ divisible by $l_0$.

For each $x\in \Sigma$, let $\mathfrak m_x\subseteq \OO_X$ denote the ideal sheaf of $x$ in $X$.
Let $\mathfrak n := \prod_{x\in \Sigma} \mathfrak m_x^{|a_x|}$.
Then the image of the natural map $\pi_*\omega_B^{\otimes l}\to \omega_X^{[l]}$ contains $\mathfrak n^l \omega_X^{[l]}$ for all positive integers $l$ divisible by $l_0$.
Thus $\Gamma(B,\omega_B^{\otimes l})$ contains the kernel of the natural map
\begin{equation*}
\Gamma(X,\omega_X^{[l]}) \to
\Gamma(X,\omega_X^{[l]}\otimes \OO_X/\mathfrak n^l) \cong \OO_X/\mathfrak n^l = \prod_{x\in \Sigma} \OO_{X,x}/\mathfrak m_x^{l|a_x|}
\label{crucial-map}
\end{equation*}
for all $l$ divisible by $l_0$.
The dimensions of the source and the target of $\phi$ both grow as polynomials of degree $d=\dim X$ in $l$.
By the asymptotic Riemann-Roch theorem, the leading coefficient of the source polynomial is $\tfrac 1{d!}  K_X^d$.
By the definition of the Hilbert-Samuel multiplicity, the target polynomial has leading coefficient $\tfrac 1{d!} \sum_{x\in \Sigma} |a_x|^d \cdot e(\OO_{X,x})$.
The result follows.
\end{proof}

\section{Singular surfaces of general type}
\label{sec:sing-surf}

In this section we apply Proposition \ref{prop:criterion-general-type} to prove Proposition \ref{prop:surf-gen-type} below, which states that certain surfaces obtained as complete intersections of quadrics are of general type.
Proposition \ref{prop:surf-gen-type} was used in the proof of the main result of this paper (Theorem \ref{thm:mainthm}). It generalizes a result proved by Tao \cite{tao} (see also \cite{SK}), who considered the case of four quadrics.
Proposition \ref{prop:surf-gen-type} proves that the set of points in the plane with rational distances to any given rational distance set with (at least) four elements, lifts to the set of rational points of a variety of general type; see the proof of Proposition \ref{prop:bounded-degree}.

We fix the following setup.
Let $(a_1,b_1),\dotsc,(a_m,b_m)$ be distinct points of $\R^2$.
Let $x,y,z,r_1,\dotsc, r_m$ be the coordinates of $\PP^{2+m}_\C$.
Let $V$ be the intersection in $\PP_\C^{2+m}$ of the quadrics defined by $r_j^2 = (x-a_j z)^2 + (y-b_jz)^2$ with $j=1,\dotsc, m$.
Finally, let $\pi : V\to \PP_\C^2$ be the projection onto the first three coordinates. Being a complete intersection of positive dimension, the surface $V$ is connected, Cohen-Macaulay and Gorenstein, that is, has invertible canonical sheaf.
Moreover, $V$ has isolated singular points, by Lemma \ref{lemma:surface-sings} below, hence is normal and irreducible.

\begin{proposition}
\label{prop:surf-gen-type}
The surface $V$ is of general type if $m\ge 4$.
\end{proposition}

To study the singularities of $V$, it is convenient to work in the following setting.
Let $X$ be a smooth surface over $\C$. Let $f_1,\dotsc,f_m\in \Gamma(X,\OO_X)$ be nonzero sections.
For each $j=1,\dotsc, m$, let $C_j\subseteq X$ be the curve defined by $f_j=0$.
Let $y_1,\dotsc,y_m$ be the coordinates of the affine space $\Aff^m_\C$. Let $Z \subseteq X\times \Aff_\C^m$ be the subscheme defined by the equations $y_j^2 = f_j$ with $j=1,\dotsc, m$, and let $\rho : Z\to X$ be the first projection.
Note that $\rho$ is finite and flat with fibers of length $2^m$, and is \'etale away from $\rho^{-1}(C_1\cup \dotsb \cup C_m)$. Thus $Z$ is smooth away from $\rho^{-1}(C_1\cup \dotsb \cup C_m)$.

The following lemma relates the geometry of the curves $C_1,\dotsc,C_m\subset X$ with the singularities of the covering surface $Z$.

\begin{lemma}
\label{lemma:cover-sings}
Let $q\in Z$ be a closed point. Let $p = \rho(q)$.
Suppose that $p \in C_1\cap \dotsb \cap C_r$, but $p\not\in C_{r+1}\cup\dotsb\cup C_m$.
\begin{enumerate}
\item Suppose that the curves $C_1,\dotsc,C_r$ are smooth at $p$ and meet pairwise transversely at $p$.
\begin{enumerate}
\item The surface $Z$ is smooth at $q$ if and only if $r\le 2$.
\item If $r\ge 3$, then $Z$ has an ordinary multiple point at $q$. The exceptional divisor of the blowup of $Z$ at $q$ is a smooth complete intersection of $r-2$ quadrics in $\PP_\C^{r-1}$.
\end{enumerate}

\item Suppose that each of the curves $C_1,\dotsc,C_r\subseteq X$ has a node  (i.e., an ordinary double point) at $p$ and that no two share a branch at $p$.
Then $Z$ has an ordinary multiple point at $q$. 
The exceptional divisor of the blowup of $Z$ at $q$ is a smooth complete intersection of $r$ quadrics in $\PP_\C^{1+r}$.
\end{enumerate}
\end{lemma}

\begin{proof}
Replacing $X$ with its \'etale cover $X'\subseteq X\times \Aff_\C^{m-r}$ defined by the equations $y_j^2 = f_j$ with $j=r+1,\dotsc,m$, we may assume that $m=r$.
Then $Z$ is smooth at $w$ if and only if the differential form on $X\times \Aff_\C^r$
\begin{equation*}
(-2y_1 dy_1 + df_1)\wedge\dotsb \wedge (-2y_r dy_r + df_r)
\end{equation*}
is nonzero at $q$.
Noting that $y_1=\dotsb = y_r=0$ at $q$, part (1a) follows.

To prove (2), suppose that each of the curves $C_1,\dotsc,C_r\subseteq X$ has a node at $p$ and that no two share a branch at $p$.
After shrinking $X$ and replacing it by an \'etale cover (or, alternatively, after replacing $\OO_{X,p}$ by its completion), we may assume that there exists a regular system of parameters $u,v\in \OO_{X,p}$ such that 
\begin{equation*}
f_j = (u + s_j v)(u+t_j v) + \epsilon_j
\end{equation*}
for all $j=1,\dotsc, r$, where $s_1,t_1,\dotsc,s_r,t_r$ are distinct complex numbers and $\epsilon_1,\dotsc,\epsilon_r \in \mathfrak m_p^3$.
It is then straightforward to compute that the exceptional divisor of the blowup of $Z$ at $q$ is the intersection of the quadrics $y_j^2 = (u + s_j v)(u+t_j v)$, where $j=1,\dotsc, r$, in the projectivized tangent space $\PP_\C^{1+r}$ of $X\times \Aff_\C^r$ at $q$.
Viewing this complete intersection as a cover of the projective line $\PP_\C^1$ with coordiantes $u,v$, it is easy to show that it is smooth.
Thus (2) holds.

Finally, to prove (1b), suppose that the curves $C_1,\dotsc,C_r$ are smooth and meet pairwise transversely at $p$, and $m\ge 3$.
We may factor the map $\rho : Z\to X$ via the subscheme $X' \subseteq X\times \Aff_\C^2$ defined by the equations $y_1^2 = f_1$ and $y_2^2= f_2$.
Let $\tilde\rho : X'\to X$ be the first projection, and
let $p'\in X'$ be the unique point lying over $p\in X$.
It is easy to show that the inverse images $\tilde\rho^{-1}C_3,\dotsc, \tilde\rho^{-1}C_r\subseteq X'$ have nodes at $p'$.
No two of them share a branch at $p'$, since both branches of $\tilde\rho^{-1}C_j$ map to the unique branch of $C_j$, for each $j=3,\dotsc,r$.
Thus (1b) follows from (2).
\end{proof}

We use the preceding lemma to describe the singularities of the surface $V$.

\begin{lemma}
\label{lemma:surface-sings}
The singularities of $V$ consist of $m2^{m-1} + 2$ ordinary multiple points.
They are the $m2^{m-1}$ points of $V$ lying over the points $(a_1,b_1,1),\dotsc,(a_m,b_m,1) \in \PP_\C^2$, and the two points of $V$ lying over $(1,\pm i,0)\in \PP_\C^2$.
The former singularities are ordinary double points.
The exceptional divisors of the blowup of $V$ at the two singularities at infinity are smooth complete intersections of $m-2$ quadrics in $\PP_\C^{m-1}$.
\end{lemma}

\begin{proof}
For each $j=1,\dotsc, m$, let $C_j \subset \PP_\C^2$ be the conic defined by the vanishing of $f_j := (x-a_j z)^2 + (y-b_jz)^2$. Writing
\begin{equation*}
f_j = ((x-a_j z) + i(y - b_j z))((x-a_j z) - i(y-b_j z)),
\end{equation*}
we see that $C_j$ is a pair of lines in meeting transversely at $(a_j,b_j,1)\in \PP_\C^2$ for each $j=1,\dotsc, j$, and that $C_2,\dotsc, C_m$ are obtained from $C_1$ by translations in the affine plane $\{z=1\}\subset \PP_\C^2$.
It follows that any two of the degenerate conics $C_j$ meet transversely at two points in the affine plane $\{z=1\}$ and at the two points $(1,\pm i,0)$.
Furthermore, no three of them meet at a point in the affine plane $\{z=1\}$.
Hence the result follows from Lemma \ref{lemma:cover-sings} with $\rho : Z\to X$ ranging over the restrictions of the projection $\pi : V\to \PP_\C^2$ to the open sets of the standard affine open cover of $\PP_\C^2$.
\end{proof}

\begin{proof}[Proof of Proposition \ref{prop:surf-gen-type}]
The surface $V$ is a complete intersection of $m$ quadrics in $\PP_\C^{2+m}$. By the adjunction formula, its canonical sheaf is
\begin{equation*}
\omega_V = \omega_{\PP^{2+m}}(2m)|_V = \OO_V(m-3),
\end{equation*}
which is ample if and only if $m\ge 4$.
Thus 
\begin{equation*}
K_V^2 = (m-3)^2 \deg V = (m-3)^2 2^m.
\end{equation*}

Let $\phi: B\to V$ be the blowup of $V$ at its singular points. Then $B$ is smooth by Lemma \ref{lemma:surface-sings}. 
Let $p\in V$ be a singular point.
Let $E$ be the exceptional divisor of $\pi$ that lies over $p$.
Let $a:= a(E,V)$ be the discrepancy of $K_V$ at $E$.
Then $K_B = \phi^* K_X + aE + F$ in $\operatorname{Pic}(B)\otimes \Q$, where $F$ is a divisor supported on $B\setminus E$.
Combining this equality with the adjunction formula $K_E = (K_B + E)|_E$, we obtain that $K_E = (a+1)E$.

In the case where $\pi(p)\in \PP_\C^2$ is one of the points $(a_1,b_1,1),\dotsc, (a_m,b_m,1)$, the singularity of $V$ at $p$ is an ordinary double point and hence canonical.
Explicitly, the exceptional divisor $E$ is a smooth conic in $\PP_\C^2$.
Hence $\omega_E = \OO_E(-1) = \OO_B(E)|_E$, which shows that $a=0$.

In the case where $\pi(p) = (1,\pm i,0)$, the exceptional divisor $E$ is a smooth complete intersection of $m-2$ quadrics in $\PP_\C^{m-1}$.
Its degree $2^{m-2}$ is equal to the Hilbert-Samuel multiplicity $e(\OO_{V,p})$. 
The canonical sheaf of $E$ is \[\omega_E = \OO_E(-m + 2(m-2)) = \OO_B((4-m)E)|_E. \]
Thus $a = 3-m<0$, so that $p$ is a non-canonical singularity in this case. The result now follows from Proposition \ref{prop:criterion-general-type}, which applies as 
$2 |3-m|^2 \cdot 2^{m-2} < (m-3)^2 2^m$ for all $m\ge 4$.
\end{proof}

\bibliographystyle{alpha}	
\bibliography{erdosulam}


\end{document}